\RequirePackage{fix-cm}
\documentclass[oneside,english]{amsart}
\usepackage[T1]{fontenc}
\usepackage[latin9]{inputenc}
\usepackage{babel}
\usepackage{enumitem}
\usepackage{amstext}
\usepackage{amsthm}
\usepackage{amssymb}
\usepackage[all]{xy}
\usepackage[pdfusetitle,
 bookmarks=true,bookmarksnumbered=false,bookmarksopen=false,
 breaklinks=false,pdfborder={0 0 1},backref=false,colorlinks=false]
 {hyperref}

\makeatletter
\numberwithin{equation}{section}
\numberwithin{figure}{section}
\theoremstyle{plain}
\newtheorem{thm}{\protect\theoremname}
\theoremstyle{plain}
\newtheorem{prop}[thm]{\protect\propositionname}
\theoremstyle{remark}
\newtheorem{rem}[thm]{\protect\remarkname}
\theoremstyle{plain}
\newtheorem{lem}[thm]{\protect\lemmaname}
\theoremstyle{definition}
\newtheorem{example}[thm]{\protect\examplename}
\theoremstyle{definition}
\newtheorem{defn}[thm]{\protect\definitionname}
\theoremstyle{plain}
\newtheorem{cor}[thm]{\protect\corollaryname}

\usepackage{xy}
\newcommand{\xyR}[1]{\xydef@\xymatrixrowsep@{#1}}
\newcommand{\xyC}[1]{\xydef@\xymatrixcolsep@{#1}}

\xyoption{2cell}
\UseAllTwocells

\newcommand{\Shv}{\mathsf{Shv}}

\newcommand{\Grp}{\mathsf{Grp}}

\newcommand{\Sch}{\mathsf{Sch}}

\newcommand{\Max}{\mathrm{Max}}
\newcommand{\WeakAss}{\mathrm{WeakAss}}

\newcommand{\Reg}{\mathrm{Reg}}

\newcommand{\Hom}{\mathrm{Hom}}
\newcommand{\Spec}{\mathrm{Spec}}

\newcommand{\Res}{\mathrm{Res}}

\makeatother

\providecommand{\corollaryname}{Corollary}
\providecommand{\definitionname}{Definition}
\providecommand{\examplename}{Example}
\providecommand{\lemmaname}{Lemma}
\providecommand{\propositionname}{Proposition}
\providecommand{\remarkname}{Remark}
\providecommand{\theoremname}{Theorem}

\begin{document}
\title{Néron Models of $\mathbb{G}_{m}$}
\author{Christophe Cornut}
\address{Sorbonne Université, Université Paris Cité, CNRS, IMJ-PRG, F-75005
Paris, France}
\begin{abstract}
We construct ``Néron models'' of $\mathbb{G}_{m}$ over any base
scheme $S$. 
\end{abstract}

\thanks{Thanks to: Alexis Bouthier, Neil Epstein, Ning Guo, Arnab Kundu, and
Joshua Mundinger. }
\maketitle

\section{Introduction}

For an arbitrary scheme $S$, we construct and analyze an exact sequence
\[
1\rightarrow\mathbb{G}_{m,S}\rightarrow\mathbb{G}_{m,S}^{lft}\rightarrow\Pi_{S}\rightarrow1
\]
of smooth commutative group schemes over $S$: $\Pi_{S}$ is the étale
group scheme associated with the Zariski sheaf of Cartier divisors
$\mathcal{D}_{S}$ on $S$, and $\mathbb{G}_{m,S}^{lft}$ is the extension
of $\Pi_{S}$ by $\mathbb{G}_{m,S}$ corresponding to the tautological
primitive line bundle $\mathcal{L}_{S}$ on $\Pi_{S}$. 

When $S$ is the spectrum of a discrete valuation domain, the group
scheme $\mathbb{G}_{m,S}^{lft}$ is the Néron lft-model of $\mathbb{G}_{m}$,
as defined in \cite[\S 10]{BoLuRa90} or \cite[B.7]{KaPr23}. A far
more general sheaf theoretical construction was sketched by Grothendieck
near the end of \cite[VIII, \S 6]{SGA7.1}. Our variant retains this
generality, but produces actual group schemes. 

When $S$ is integral with generic point $\eta$, the generic fiber
of $\mathbb{G}_{m,S}^{lft}$ equals $\mathbb{G}_{m,\eta}$ and restriction
to the generic fibers thus induces a morphism of group sheaves
\[
\theta:\mathbb{G}_{m,S}^{lft}\rightarrow\Res_{\eta/S}\left(\mathbb{G}_{m,\eta}\right).
\]
We then say that a flat morphism $f:T\rightarrow S$ is \emph{convex
}when the restriction of $\theta$ to the small Zariski site of $T$
is an isomorphism. This captures the Néron property: if $f$ is flat
and convex, the restriction map $\mathbb{G}_{m,S}^{lft}(T)\rightarrow\mathbb{G}_{m}(T_{\eta})$
is an isomorphism, or equivalently, any element of $\mathbb{G}_{m}(T_{\eta})$
extends uniquely to a section of $\mathbb{G}_{m,S}^{lft}$ over $T$.
When $S$ is a Prüfer scheme (in the weak sense that all of its local
rings are valuation domains), we show that any smooth $f:T\rightarrow S$
is (flat and) convex. 

It follows that for any finitely presented smooth group scheme $G$
over any base scheme $S$, the sheaf $X(G)$ of characters of $G$
satisfies the valuative criterion of properness. In particular, if
$S$ is the spectrum of a valuation domain $\mathcal{O}$ with fraction
field $K$, the image of $G(\mathcal{O})\rightarrow G(K)$ is contained
in the kernel of 
\[
\nu_{G}:G(K)\rightarrow\Hom(X(G_{K}),\Gamma),\qquad g\mapsto(\chi\mapsto v(\chi(g))
\]
where $\Gamma=K^{\times}/\mathcal{O}^{\times}$ is the divisibility
group and $v:K^{\times}\rightarrow\Gamma$ is the valuation. 

\section{The construction}

\subsection{Regular sections}

For an open $U$ of $S$ and a section $x\in\mathcal{O}_{S}(U)$,
the following conditions are equivalent: 
\begin{itemize}
\item Multiplication by $x$ gives a monomorphism $\mathcal{O}_{U}\rightarrow\mathcal{O}_{U}$;
\item There is an affine cover $U=\cup U_{i}$ such that $x_{U_{i}}$ is
regular in $\mathcal{O}_{S}(U_{i})$;
\item For every affine open $V$ of $U$, $x_{V}$ is regular in $\mathcal{O}_{S}(V)$;
\item For every $s\in U$, $x_{s}$ is regular in $\mathcal{O}_{S,s}$.
\end{itemize}
We then say that $x$ is a \emph{regular section} of $\mathcal{O}_{S}$
over $U$, and denote by $\mathcal{O}_{S}^{reg}(U)$ the set of all
such sections. It is a saturated submonoid of the cancellative submonoid
$\Reg(\mathcal{O}_{S}(U))$ of regular elements in the multiplicative
monoid $(\mathcal{O}_{S}(U),\times)$, and all three monoids share
the same subgroup $\mathcal{O}_{S}^{\times}(U)$ of invertible elements.
The assignment $U\mapsto\mathcal{O}_{S}^{reg}(U)$ defines a multiplicative
submonoid sheaf $\mathcal{O}_{S}^{reg}$ of $(\mathcal{O}_{S},\times)$,
whereas $U\mapsto\Reg\left(\mathcal{O}_{S}(U)\right)$ may not even
be a subpresheaf of $\mathcal{O}_{S}$. Likewise, the stalk $\mathcal{O}_{S,s}^{reg}$
of $\mathcal{O}_{S}^{reg}$ at $s\in S$ is a saturated submonoid
of the cancellative submonoid $\Reg(\mathcal{O}_{S,s})$ of regular
elements in the multiplicative monoid $(\mathcal{O}_{S,s},\times)$,
and all three monoids again share the same subgroup $\mathcal{O}_{S,s}^{\times}$
of invertible elements. 

\subsection{Germs of meromorphic functions}

We denote by $\mathcal{M}_{S}$ the sheaf of $\mathcal{O}_{S}$-algebras
associated with the presheaf $U\mapsto\mathcal{O}_{S}(U)[\mathcal{O}_{S}^{reg}(U)^{-1}]$.
Thus $\mathcal{M}_{S}$ is the sheaf of \emph{germs of meromorphic
functions} on $S$ considered in \cite[20.1.3]{EGA4.4}, except that
the definition of regular sections given there has to be corrected
as indicated above, cf.~\cite{Kl79}. The stalk $\mathcal{M}_{S,s}$
of $\mathcal{M}_{S}$ at $s\in S$ is the sub-$\mathcal{O}_{S,s}$-algebra
$\mathcal{O}_{S,s}[(\mathcal{O}_{S,s}^{reg})^{-1}]$ of the total
ring of fractions $\mathcal{Q}(\mathcal{O}_{S,s})=\mathcal{O}_{S,s}[\Reg(\mathcal{O}_{S,s})^{-1}]$
of $\mathcal{O}_{S,s}$. In $\mathcal{Q}(\mathcal{O}_{S,s})$, 
\[
\mathcal{O}_{S,s}^{reg}=\mathcal{O}_{S,s}\cap\mathcal{M}_{S,s}^{\times}\qquad\text{and}\qquad\Reg(\mathcal{O}_{S,s})=\mathcal{O}_{S,s}\cap\mathcal{Q}(\mathcal{O}_{S,s})^{\times}.
\]
The first equality implies $\mathcal{O}_{S}^{reg}=\mathcal{O}_{S}\cap\mathcal{M}_{S}^{\times}$
in $\mathcal{M}_{S}$. Since $\mathcal{O}_{S,s}^{\times}=\mathcal{O}_{S,s}^{reg}\cap(\mathcal{O}_{S,s}^{reg})^{-1}$
in $\mathcal{M}_{S,s}^{\times}$ for all $s\in S$, we also have $\mathcal{O}_{S}^{\times}=\mathcal{O}_{S}^{reg}\cap(\mathcal{O}_{S}^{reg})^{-1}$
in $\mathcal{M}_{S}^{\times}$. One checks that the presheaf $U\mapsto\mathcal{O}_{S}(U)[\mathcal{O}_{S}^{reg}(U)^{-1}]$
is separated; thus for any open $U$ of $S$,
\[
\xyR{1.2pc}\xymatrix{ & \mathcal{O}_{S}(U)\ar@{^{(}->}[d]\\
\mathcal{Q}(\mathcal{O}_{S}(U)) & \mathcal{O}_{S}(U)[\mathcal{O}_{S}^{reg}(U)^{-1}]\ar@{_{(}->}[l]\ar@{^{(}->}[r] & \mathcal{M}_{S}(U).
}
\]
When $U$ is affine, $\mathcal{O}_{S}^{reg}(U)=\Reg(\mathcal{O}_{S}(U))$
and $\mathcal{O}_{S}(U)[\mathcal{O}_{S}^{reg}(U)^{-1}]=\mathcal{Q}(\mathcal{O}_{S}(U))$;
this then gives an embedding $\mathcal{Q}(\mathcal{O}_{S}(U))\hookrightarrow\mathcal{M}_{S}(U)$,
but it may fail to be an isomorphism, and such a morphism usually
does not exist for a non-affine $U$. 

\subsection{Simplifications }

Most of these pathologies disappear under good assumptions on $S$.
Suppose for instance that $S$ is integral (i.e.~reduced and irreducible)
with generic point $\eta$ and fraction field $k(S)=k(\eta)=\mathcal{O}_{S,\eta}$.
Then for any open $U\neq\emptyset$ of $S$, $\mathcal{O}_{S}(U)$
is an integral domain, $\mathcal{O}_{S}^{reg}(U)=\Reg(\mathcal{O}_{S}(U))=\mathcal{O}_{S}(U)\setminus\{0\}$,
so $U\mapsto\Reg(\mathcal{O}_{S}(U))$ is a sheaf, $U\mapsto\mathcal{Q}(\mathcal{O}_{S}(U))$
is a separated presheaf, and $\mathcal{M}_{S}$ is the associated
sheaf; if $U$ is affine, then $\mathcal{Q}(\mathcal{O}_{S}(U))$
equals $k(S)$, so $\mathcal{M}_{S}$ is the constant sheaf associated
with $k(S)$, and for any $s\in S$, $\mathcal{M}_{S,s}=\mathcal{Q}(\mathcal{O}_{S,s})=k(S)$. 

More generally, we have the following proposition. Let $\Max(S)$
be the set of maximal points of $S$, aka the generic points of the
irreducible components of $S$, and let $\WeakAss(S)$ be the set
of weakly associated points of $S$. 
\begin{prop}
\label{prop:NoPathologies}With notations as above:
\begin{enumerate}
\item $\Max(S)\subset\WeakAss(S)$ with equality when $S$ is reduced.
\item For any open $U$ of $S$ and $f\in\mathcal{O}_{S}(U)$, 
\[
f\in\mathcal{O}_{S}^{reg}(U)\iff\forall s\in\WeakAss(U),\quad f(s)\neq0\quad\text{in }k(s).
\]
\item If $\WeakAss(S)$ is locally finite at $s\in S$, then 
\[
\mathcal{O}_{S,s}^{reg}=\Reg(\mathcal{O}_{S,s})\quad\text{and}\quad\mathcal{M}_{S,s}=\mathcal{Q}(\mathcal{O}_{S,s}).
\]
\item If $\Max(S)$ is locally finite in $S$ and equal to $\WeakAss(S)$,
then 
\begin{enumerate}
\item $\mathcal{M}_{S}$ is quasi-coherent;
\item For any affine open $U$ of $S$, $\mathcal{M}_{S}(U)=\mathcal{Q}(\mathcal{O}_{S}(U));$
\item For any open $U$ of $S$, $\mathcal{M}_{S}(U)=\prod_{\eta\in\Max(U)}\mathcal{O}_{S,\eta}$;
\item For any $s\in S$, $\mathcal{M}_{S,s}=\mathcal{Q}(\mathcal{O}_{S,s})=\prod_{\eta\in\Max(S),s\in\overline{\eta}}\mathcal{O}_{S,\eta}$. 
\end{enumerate}
\end{enumerate}
\end{prop}

\begin{proof}
For $(1)$, see~\cite[\href{https://stacks.math.columbia.edu/tag/05AQ}{Tags 05AQ} \& \href{https://stacks.math.columbia.edu/tag/0EME}{0EME}]{SP}.
For $(2)$, suppose $f\in\mathcal{O}_{S}^{reg}(U)$, so that $f_{s}\in\Reg(\mathcal{O}_{S,s})$
for all $s\in U$; if $s\in\WeakAss(U)$, $\Reg(\mathcal{O}_{S,s})=\mathcal{O}_{S,s}^{\times}$,
so $f(s)\neq0$ in $k(s).$ Conversely, suppose that $f(s)\neq0$
in $k(s)$ for all $s\in\WeakAss(U)$. Then $f\in\mathcal{O}_{S}^{reg}(U)$
by~\cite[\href{https://stacks.math.columbia.edu/tag/0AVP}{Tag 0AVP}]{SP}.
For $(3)$, we have to show that if $\WeakAss(S)$ is locally finite
at $s\in S$ and $\alpha\in\Reg(\mathcal{O}_{S,s})$, then $\alpha\in\mathcal{O}_{S,s}^{reg}$.
By assumption, we may lift $\alpha$ to $a\in A=\mathcal{O}_{S}(U)$
for some affine open neighborhood $U=\Spec(A)$ of $s$ in $S$ such
that the set $\mathcal{W}$ of weakly associated primes of $A$ is
finite. Let $p$ be the prime of $A$ corresponding to $s\in U$.
Note that $\alpha\in\mathcal{O}_{S,s}^{reg}$ if and only if there
is an $f\in A$ with $f\notin p$ such that the image $a_{f}$ of
$a$ in $A_{f}$ is regular. Let $\mathcal{F}$ be the set of all
such $f$'s; we have to show that $\mathcal{F}\neq\emptyset$. By~\cite[\href{https://stacks.math.columbia.edu/tag/05C3}{Tag 05C3}]{SP}
and for any $f\in A$, 
\[
\Reg(A_{p})=A_{p}\setminus\cup_{q\in\mathcal{W}_{p}}qA_{p}\quad\text{and}\quad\Reg(A_{f})=A_{f}\setminus\cup_{q\in\mathcal{W}_{f}}qA_{f}
\]
where $\mathcal{W}_{p}=\{q\in\mathcal{W}:q\subset p\}$ and $\mathcal{W}_{f}=\{q\in\mathcal{W}:f\notin q\}$.
Since $\alpha\in\Reg(A_{p})$, $a\not\in\cup_{q\in\mathcal{W}_{p}}q$
and our $f$'s are those for which $f\notin p$ and $a\notin\cup_{q\in\mathcal{W}_{f}}q$.
So 
\[
\mathcal{F}=\cap_{q\in\mathcal{W}(a)}q\setminus p\quad\text{where}\quad\mathcal{W}(a)=\{q\in\mathcal{W}:a\in q\}.
\]
Since no $q\in\mathcal{W}_{p}$ contains $a$, no $q\in\mathcal{W}(a)$
is contained in $p$; since $\mathcal{W}(a)$ is finite, it follows
that $\cap_{q\in\mathcal{W}(a)}q\not\subset p$, whence $\mathcal{F}\neq\emptyset$,
which finishes the proof of $(3)$. 

Suppose now that $\Max(S)$ is locally finite in $S$ and equal to
$\WeakAss(S)$. The conjunction of $(a)$ and $(b)$ is equivalent
to the following claim: for any affine open $U=\Spec(A)$ of $S$,
the morphism of $\mathcal{O}_{U}$-algebras $\varphi:\mathcal{Q}(A)^{\sim}\rightarrow\mathcal{M}_{S}\vert_{U}$
induced by the embedding of $\mathcal{Q}(A)=\mathcal{Q}(\mathcal{O}_{S}(U))=\mathcal{O}_{S}(U)[\mathcal{O}_{S}^{reg}(U)^{-1}]$
into $\mathcal{M}_{S}(U)$ is an isomorphism. This can be checked
on stalks. So let $s\in U$ correspond to a prime $p$ of $A$. By
$(3)$, $\mathcal{M}_{S,s}=\mathcal{Q}(\mathcal{O}_{S,s})=\mathcal{Q}(A_{p})$,
and $\varphi_{s}:\mathcal{Q}(A)_{s}^{\sim}\rightarrow\mathcal{M}_{S,s}$
is the $A_{p}$-algebra morphism $\mathcal{Q}(A)_{p}\rightarrow\mathcal{Q}(A_{p})$
extending $A\rightarrow A_{p}\hookrightarrow\mathcal{Q}(A_{p})$.
By our assumptions on $S$, the set $\mathcal{M}$ of minimal primes
of $A$ if finite and $\cup_{q\in\mathcal{M}}q$ is the set of zerodivisors
of $A$; thus by \cite[\href{https://stacks.math.columbia.edu/tag/02LX}{Tag 02LX}]{SP},
$\mathcal{Q}(A)=\prod_{q\in\mathcal{M}}A_{q}$; likewise, $\mathcal{Q}(A_{p})=\prod_{q\in\mathcal{M}_{p}}A_{q}$,
where $\mathcal{M}_{p}=\{q\in\mathcal{M}:q\subset p\}$. We thus have
to show that for any $q\in\mathcal{M}$, the localization $A_{q,p}$
of $A_{q}$ at $p$ is $A_{q}$ if $q\subset p$ and $0$ otherwise.
If $q\subset p$, the elements of $A\setminus p\subset A\setminus q$
are already invertible in $A_{q}$, so $A_{q,p}=A_{q}$; if $q\not\subset p$,
any element of $q\setminus p$ is invertible in $A_{p}$ and nilpotent
in $A_{q}$ (since $q$ is minimal), whence simultaneously invertible
and nilpotent in $A_{q,p}=A_{q}\otimes_{A}A_{p}$, so $A_{q,p}=0$.
Thus $\mathcal{Q}(A)_{p}\rightarrow\mathcal{Q}(A_{p})$ is an isomorphism,
which proves $(a)$ and $(b)$, along with $(c)$ for affine $U$'s,
from which $(d)$ easily follows. The general formula in $(c)$ defines
a sheaf $\mathcal{M}_{S}^{\prime}$ on $S$, and the formula in $(d)$
then gives a morphism $\mathcal{M}_{S}\rightarrow\mathcal{M}_{S}^{\prime}$
which is an isomorphism on stalks, whence an isomorphism, and this
proves $(c)$ for arbitrary $U$'s. 
\end{proof}
\begin{rem}
Under the assumptions of $(4)$, the quasi-coherent sheaf $\mathcal{M}_{S}$
is flasque by $(c)$, and so is the Zariski sheaf $\mathcal{M}_{S}^{\times}.$
Moreover, the affine morphism $\Spec(\mathcal{M}_{S})\rightarrow S$
is the embedding of $\coprod_{\eta\in\Max(S)}\Spec(\mathcal{O}_{S,\eta})\simeq\Max(S)$
as a discrete subspace of $S$. 
\end{rem}

\subsection{The Zariski sheaf of Cartier divisors}

It is defined by the exact sequence 
\[
\xymatrix{1\ar[r] & \mathcal{O}_{S}^{\times}\ar[r] & \mathcal{M}_{S}^{\times}\ar[r]^{\mathrm{div}} & \mathcal{D}_{S}\ar[r] & 1}
\]
of Zariski sheaves on $S$. It is a sheaf of partially ordered commutative
groups, where the subsheaf of non-negative elements $\mathcal{D}_{S}^{+}$
of $\mathcal{D}_{S}$ is the image of $\mathcal{O}_{S}^{reg}\subset\mathcal{M}_{S}^{\times}$
in $\mathcal{D}_{S}$ \cite[21.1]{EGA4.4}. The stalk of $\mathcal{D}_{S}$
at $s\in S$ is the subgroup $\mathcal{D}_{S,s}=\mathcal{M}_{S,s}^{\times}/\mathcal{O}_{S,s}^{\times}$
of the divisibility group $\mathcal{D}(\mathcal{O}_{S,s})=\mathcal{Q}(\mathcal{O}_{S,s})^{\times}/\mathcal{O}_{S,s}^{\times}$
of $\mathcal{O}_{S,s}$. Since $\mathcal{O}_{S,s}^{reg}=\mathcal{M}_{S}^{\times}\cap\Reg(\mathcal{O}_{S,s})$
in $\mathcal{Q}(\mathcal{O}_{S,s})^{\times}$, $\mathcal{D}_{S,s}^{+}=\mathcal{D}(\mathcal{O}_{S,s})^{+}\cap\mathcal{D}_{S,s}$,
i.e.~$\mathcal{D}_{S,s}\hookrightarrow\mathcal{D}(\mathcal{O}_{S,s})$
is an \emph{embedding} of partially ordered commutative groups; since
$\mathcal{O}_{S,s}^{reg}$ is also saturated in $\Reg(\mathcal{O}_{S,s})$,
$\mathcal{D}_{S,s}$ is a \emph{convex} subgroup of $\mathcal{D}(\mathcal{O}_{S,s})$,
i.e.~$\mathcal{D}_{S,s}\hookrightarrow\mathcal{D}(\mathcal{O}_{S,s})$
is a \emph{convex embedding}. 

We write $\text{div}:\mathcal{M}_{S}^{\times}\rightarrow\mathcal{D}_{S}$
for the quotient map, and use additive notations for the group structure
on $\mathcal{D}_{S}$. Mapping $x\in\mathcal{M}_{S}^{\times}(U)$
to the free rank one $\mathcal{O}_{U}$-submodule of $\mathcal{M}_{S}\vert_{U}=\mathcal{M}_{U}$
spanned by $x$ induces an isomorphism $d\mapsto\mathcal{L}(d)$ between
$\mathcal{D}_{S}$ and the sheaf of (fractional) invertible $\mathcal{O}_{S}$-submodules
in $\mathcal{M}_{S}$, cf.~\cite[21.2]{EGA4.4}. 
\begin{rem}
\label{rem:EdgeMorph4Specializations}For a specialization $s'\rightsquigarrow s$
in $S$, the snake lemma gives an exact sequence
\[
\xyC{1pc}\xymatrix{1\ar[r] & \ker\left(\mathcal{O}_{S,s}^{\times}\rightarrow\mathcal{O}_{S,s'}^{\times}\right)\ar[r] & \ker\left(\mathcal{M}_{S,s}^{\times}\rightarrow\mathcal{M}_{S,s'}^{\times}\right)\ar[r] & \ker\left(\mathcal{D}_{S,s}\rightarrow\mathcal{D}_{S,s'}\right)\ar[d]^{\delta_{s,s'}}\\
1 & \mathrm{coker}\left(\mathcal{D}_{S,s}\rightarrow\mathcal{D}_{S,s'}\right)\ar[l] & \mathrm{coker}\left(\mathcal{M}_{S,s}^{\times}\rightarrow\mathcal{M}_{S,s'}^{\times}\right)\ar[l] & \mathrm{coker}\left(\mathcal{O}_{S,s}^{\times}\rightarrow\mathcal{O}_{S,s'}^{\times}\right)\ar[l]
}
\]
If $S$ is locally integral, this reduces to an isomorphism 
\[
\delta_{s,s'}:\ker\left(\mathcal{D}_{S,s}\rightarrow\mathcal{D}_{S,s'}\right)\rightarrow\mathcal{O}_{S,s'}^{\times}/\mathcal{O}_{S,s}^{\times}.
\]
\end{rem}

\subsection{The group scheme of Cartier divisors}

Let $\pi:\Pi_{S}\rightarrow S$ be the \emph{espace étalé} of $\mathcal{D}_{S}$.
Thus $\Pi_{S}$ is a partially ordered commutative group scheme over
$S$, whose structure morphism $\pi$ is a local isomorphism -- in
particular, $\Pi_{S}$ is étale over $S$. As a set, $\left|\Pi_{S}\right|=\coprod_{s\in S}\mathcal{D}_{S,s}$
with the obvious projection to $S$. As a scheme, it is covered by
opens $U_{d}=\{d_{s},s\in U\}\simeq U$ for opens $U$ of $S$ and
sections $d\in\mathcal{D}_{S}(U)$. The structural sheaf $\mathcal{O}_{\Pi_{S}}$
is the sheaf theoretical inverse image $\pi^{-1}\mathcal{O}_{S}$
of $\mathcal{O}_{S}$ under the topological map $\pi$. Thus for any
sheaf of $\mathcal{O}_{S}$-module $\mathcal{F}$ on $S$, 
\[
\pi^{\ast}\mathcal{F}=\mathcal{O}_{\Pi_{S}}\otimes_{\pi^{-1}\mathcal{O}_{S}}\pi^{-1}\mathcal{F}=\pi^{-1}\mathcal{F}.
\]
In particular, $\mathcal{F}\mapsto\pi^{\ast}\mathcal{F}$ is exact.
The group structure on $\Pi_{S}$ is given set theoretically by the
coproduct of the sum maps on the stalks of $\mathcal{D}_{S}$, 
\[
\left|\Pi_{S}\times_{S}\Pi_{S}\right|=\coprod_{s}\mathcal{D}_{S,s}\times\mathcal{D}_{S,s}\rightarrow\coprod_{s}\mathcal{D}_{S,s}=\left|\Pi_{S}\right|.
\]
The poset structure is given by the submonoid $\Pi_{S}^{+}\subset\Pi_{S}$
with $\left|\Pi_{S}^{+}\right|=\coprod_{s\in S}\mathcal{D}_{S,s}^{+}$. 
\begin{rem}
\label{rem:PointsOfPiSOverLocals}For a local $S$-scheme $T$ with
closed point $t\in T$ above $s\in S$, the map $a\mapsto a(t)$ identifies
$\Pi_{S}(T)$ with the group $\mathcal{D}_{S,s}=\mathcal{M}_{S,s}^{\times}/\mathcal{O}_{S,s}^{\times}$. 
\end{rem}

\subsection{The primitive line bundle $\mathcal{L}$ on $\Pi_{S}$}

There is a unique invertible subsheaf $\mathcal{L}$ of $\pi^{\ast}\mathcal{M}_{S}$
such that for any open $U$ of $S$ and $d\in\mathcal{D}_{S}(U)$,
the restriction of $\mathcal{L}$ to the open $U_{d}$ of $\Pi_{S}$
is equal to the restriction of the subsheaf $\pi^{\ast}\mathcal{L}(d)$
of $\pi^{\ast}\mathcal{M}_{U}$ to the open $U_{d}$ of $\pi^{-1}(U)\subset\Pi_{S}$.
Indeed, uniqueness is clear; for the existence, we have to show that
for sections $d$ and $e$ of $\mathcal{D}_{S}$ over opens $U$ and
$V$ of $S$, $\pi^{\ast}\mathcal{L}(d)$ and $\pi^{\ast}\mathcal{L}(e)$
agree on $U_{d}\cap V_{e}$. Now $U_{d}\cap V_{e}=W_{f}$, where $W$
is the largest open of $U\cap V$ where $d$ and $e$ agree, and $f=d_{W}=e_{W}$
in $\mathcal{D}_{S}(W)$; thus $\mathcal{L}(d)\vert_{W}=\mathcal{L}(f)=\mathcal{L}(e)\vert_{W}$
in $\mathcal{M}_{W}$, so that $\pi^{\ast}\mathcal{L}(d)$ and $\pi^{\ast}\mathcal{L}(e)$
even agree on the larger open $\pi^{-1}(W)$ of $\Pi_{S}$.

Let $p_{1},p_{2}$ and $m:\Pi_{S}\times_{S}\Pi_{S}\rightarrow\Pi_{S}$
be the projections and multiplication maps, and set $\overline{\pi}=\pi\circ p_{1}=\pi\circ p_{2}=\pi\circ m$
for the structural morphism $\Pi_{S}\times_{S}\Pi_{S}\rightarrow S$.
All of these maps are local isomorphisms. The $\mathcal{O}_{S}$-bilinear
multiplication map of $\mathcal{M}_{S}$ induces an isomorphism $\mathcal{M}_{S}\otimes_{\mathcal{O}_{S}}\mathcal{M}_{S}\rightarrow\mathcal{M}_{S}$,
which pulls-back to an isomorphism 
\[
p_{1}^{\ast}\left(\pi^{\ast}\mathcal{M}_{S}\right)\otimes p_{2}^{\ast}\left(\pi^{\ast}\mathcal{M}_{S}\right)=\overline{\pi}^{\ast}\mathcal{M}_{S}\otimes\overline{\pi}^{\ast}\mathcal{M}_{S}\rightarrow\overline{\pi}^{\ast}\mathcal{M}_{S}=m^{\ast}\left(\pi^{\ast}\mathcal{M}_{S}\right).
\]
The latter induces a canonical isomorphism $p_{1}^{\ast}\mathcal{L}\otimes p_{2}^{\ast}\mathcal{L}\simeq m^{\ast}\mathcal{L}$.
Indeed for sections $d$ and $e$ of $\mathcal{D}_{S}$ over opens
$U$ and $V$ of $S$, $m$ induces an isomorphism $U_{d}\times_{S}V_{e}\rightarrow W_{f}$
where $W=U\cap V$ and $f=d_{W}+e_{W}$, and our claim thus follows
from the fact that 
\[
\mathcal{L}(d_{W})\cdot\mathcal{L}(e_{W})=\mathcal{L}(f)\quad\text{in}\quad\mathcal{M}_{W}.
\]
In other words, $\mathcal{L}$ is a \emph{primitive} line bundle on
$\Pi_{S}$.

\subsection{The smooth group scheme $\mathbb{G}_{m,S}^{lft}$}

It follows that $\mathcal{L}$ gives rise to an extension 
\[
1\rightarrow\mathbb{G}_{m,S}\rightarrow\mathbb{G}_{m,S}^{lft}\rightarrow\Pi_{S}\rightarrow1
\]
of (say) Zariski sheaves on $\Sch_{S}$ as follows. For an $S$-scheme
$T$, let $\mathbb{G}_{m,S}^{lft}(T)$ be the set of pairs $(a,\alpha)$
where $a\in\Pi_{S}(T)$ and $\alpha$ is an invertible section of
$a^{\ast}\mathcal{L}$ over $T$, i.e.~an element $\alpha\in\Gamma(T,a^{\ast}\mathcal{L})$
such that the induced morphism $\mathcal{O}_{T}\rightarrow a^{\ast}\mathcal{L}$
is an isomorphism. The product $(c,\gamma)=(a,\alpha)\cdot(b,\beta)$
of two such pairs is defined by $c=a+b$ in $\Pi_{S}(T)$, and $\gamma=\alpha\otimes\beta$
in $a^{\ast}\mathcal{L}\otimes b^{\ast}\mathcal{L}\simeq c^{\ast}\mathcal{L}$.
The unit is $(0,1)$, where $1$ is here the unit of $0^{\ast}\mathcal{L}=\mathcal{O}_{T}$,
and the inverse of $(a,\alpha)$ is $(-a,\alpha^{-1})$, where $\alpha^{-1}$
is the unique invertible section of $(-a)^{\ast}\mathcal{L}$ such
that $\alpha\otimes\alpha^{-1}\mapsto1$ under the isomorphism $a^{\ast}\mathcal{L}\otimes(-a)^{\ast}\mathcal{L}\rightarrow0^{\ast}\mathcal{L}=\mathcal{O}_{T}$.
These constructions are plainly functorial in $T$, and define a Zariski
sheaf $\mathbb{G}_{m,S}^{lft}$ on $\Sch_{S}$. The formulae $(a,\alpha)\mapsto a$
and $\lambda\mapsto(0,\lambda)$ respectively define a (Zariski) surjective
morphism $\mathbb{G}_{m,S}^{lft}\rightarrow\Pi_{S}$ and its kernel
$\mathbb{G}_{m,S}\hookrightarrow\Pi_{S}$. Being a $\mathbb{G}_{m}$-torsor
over $\Pi_{S}$, $\mathbb{G}_{m,S}^{lft}$ is representable by a smooth
$\Pi_{S}$-scheme, and since $\Pi_{S}\rightarrow S$ itself is étale,
$\mathbb{G}_{m,S}^{lft}$ is a smooth group scheme over $S$, and
$\mathbb{G}_{m,S}\hookrightarrow\mathbb{G}_{m,S}^{lft}$ is an open
immersion which identifies $\mathbb{G}_{m,S}$ with the neutral component
of $\mathbb{G}_{m,S}^{lft}$, and $\Pi_{S}$ with its group of connected
components. 
\begin{rem}
\label{rem:PointsOfGmLFToverLocals}For a local $S$-scheme $T$ with
closed point $t\in T$ above $s\in S$, elements $(a,\alpha)$ of
$\mathbb{G}_{m,S}^{lft}(T)$ correspond to pairs $(d,\delta)$ where
$d\in\mathcal{D}_{S,s}$ and $\delta$ is an $\mathcal{O}_{T,t}$-basis
of $\mathcal{L}(d)\otimes_{\mathcal{O}_{S,s}}\mathcal{O}_{T,t}$;
here $\mathcal{L}(d)$ is the free $\mathcal{O}_{S,s}$-submodule
of $\mathcal{M}_{S,s}=\mathcal{O}_{S,s}[(\mathcal{O}_{S,s}^{reg})^{-1}]$
spanned by any lift of $d\in\mathcal{M}_{S,s}^{\times}/\mathcal{O}_{S,s}^{\times}$.
The corresponding morphism $a:T\rightarrow\Pi_{S}$ maps $t$ to $d\in\pi_{S}^{-1}(s)=\mathcal{D}_{S,s}$,
with $\alpha_{t}=\delta$ in $(a^{\ast}\mathcal{L})_{t}=\mathcal{L}(d)\otimes_{\mathcal{O}_{S,s}}\mathcal{O}_{T,t}$. 
\end{rem}

\section{Separation}

We say that the scheme $S$ is \emph{locally coherent} if it has a
covering $S=\cup_{i}U_{i}$ by affine opens $U_{i}=\Spec(A_{i})$
such that each $A_{i}$ is a coherent ring; equivalently, for every
affine open $U=\Spec(A)$ of $S$, $A$ is coherent. 
\begin{lem}
\label{lem:GmlftQuasiSepIfSloccoh}If $S$ is locally coherent, then
$\Pi_{S}$ and $\mathbb{G}_{m,S}^{lft}$ are quasi-separated over
$S$. 
\end{lem}

\begin{proof}
Since $\mathbb{G}_{m,S}^{lft}\rightarrow\Pi_{S}$ is affine, it is
sufficient to establish that $\Pi_{S}\rightarrow S$ is quasi-separated,
which amounts to showing that its unit section is quasi-compact, or
that its image $S_{0}$ is a \emph{retro-compact} open of $\Pi_{S}$.
So, for an affine open $U=\Spec(\mathcal{O})$ of $S$ and $d\in\mathcal{D}_{S}(U)$,
we have to show that $S_{0}\cap U_{d}$ is quasi-compact; this is
isomorphic to the open of $U$ where $d\equiv0$, which is the intersection
of the opens $U(\pm d\geq0)$ where $\pm d\geq0$, so it is sufficient
to establish that $U(d\geq0)$ is quasi-compact, and we may assume
that $d$ is principal, say $d=\mathrm{div}(a/b)$ with $a,b\in\Reg(\mathcal{O})$.
Then $p\in\Spec(\mathcal{O})$ belongs to $U(d\geq0)$ if and only
there is an $f\in\mathcal{O}\setminus p$ such that $b\mid fa$ in
$\mathcal{O}$, i.e.~if and only if $(\mathcal{O}b:\mathcal{O}a)=\{\lambda\in\mathcal{O}:\lambda a\in\mathcal{O}b\}$
is not contained in $p$. Since $\mathcal{O}$ is coherent by assumption,
$(\mathcal{O}b:\mathcal{O}a)$ is finitely generated. Thus $U(d\geq0)$
is the union of finitely many principal opens of $U$, whence indeed
quasi-compact. 
\end{proof}
\begin{lem}
Let $A$ be an $S$-valuation ring with fraction field $K$. Then
the kernel of $\mathbb{G}_{m,S}^{lft}(A)\rightarrow\mathbb{G}_{m,S}^{lft}(K)$
is canonically isomorphic to the kernel of 
\[
\ker\left(\mathcal{D}_{S,s}\rightarrow\mathcal{D}_{S,s'}\right)\stackrel{\delta_{s,s'}}{\longrightarrow}\mathrm{coker}\left(\mathcal{O}_{S,s}^{\times}\rightarrow\mathcal{O}_{S,s'}^{\times}\right)\stackrel{\mathrm{can}}{\longrightarrow}\mathrm{coker}\left(A^{\times}\rightarrow K^{\times}\right)
\]
where $s'$ and $s\in S$ are the images of the generic and closed
points of $\Spec(A)$, while $\delta_{s,s'}$ is the edge morphism
associated to $s'\rightsquigarrow s$ in Remark~\ref{rem:EdgeMorph4Specializations}.
\end{lem}

\begin{proof}
Since $A$ and $K$ are local, we have a commutative diagram of exact
sequences
\[
\xymatrix{1\ar[r] & \mathbb{G}_{m}(A)\ar[r]\ar[d] & \mathbb{G}_{m,S}^{lft}(A)\ar[r]\ar[d]^{\star} & \Pi_{S}(A)\ar[r]\ar[d] & 1\\
1\ar[r] & \mathbb{G}_{m}(K)\ar[r] & \mathbb{G}_{m,S}^{lft}(K)\ar[r] & \Pi_{S}(K)\ar[r] & 1
}
\]
The snake lemma then identifies the kernel of $\star$ with the kernel
of the edge map
\[
\delta:\ker\left(\Pi_{S}(A)\rightarrow\Pi_{S}(K)\right)\rightarrow\mathrm{coker}\left(A^{\times}\rightarrow K^{\times}\right).
\]
On the other hand, $\Pi_{S}(A)\rightarrow\Pi_{S}(K)$ is the localization
map $\mathcal{D}_{S,s}\rightarrow\mathcal{D}_{S,s'}$ by Remark~\ref{rem:PointsOfPiSOverLocals}.
The fact that $\delta$ factors as indicated can be seen by replacing
$A\hookrightarrow K$ by $\mathcal{O}_{S,s}\rightarrow\mathcal{O}_{S,s'}$
in the above diagram, and using Remark~\ref{rem:PointsOfGmLFToverLocals}
to identify the resulting exact sequences with those used in the definition
of $\delta_{s,s'}$. 
\end{proof}
\begin{prop}
\label{prop:SeparatedCase}Suppose that $S$ is locally coherent with
uniserial local rings. Then 
\[
\mathbb{G}_{m,S}^{lft}\text{ is a \emph{separated} \ensuremath{S}-group scheme.}
\]
\end{prop}

\begin{proof}
Since $S$ is locally coherent, $\mathbb{G}_{m,S}^{lft}\rightarrow S$
is quasi-separated by lemma~\ref{lem:GmlftQuasiSepIfSloccoh}. By
the valuative criterion of separatedness, it remains to establish
that for any $S$-valuation ring $A$ with fraction field $K$, the
kernel of $\mathbb{G}_{m,S}^{lft}(A)\rightarrow\mathbb{G}_{m,S}^{lft}(K)$
is trivial. Let $s'$ and $s\in S$ be the images of the generic and
closed points of $\Spec(A)$. 

We first show that the edge map $\delta_{s,s'}$ from the previous
lemma is injective, i.e.~that any element $x\in\ker(\mathcal{M}_{S,s}^{\times}\rightarrow\mathcal{M}_{S,s'}^{\times})$
belongs to $\mathcal{O}_{S,s}^{\times}$. Since $x\in\mathcal{M}_{S,s}^{\times}$,
we may write $x=a/b$ with $a,b\in\mathcal{O}_{S,s}^{reg}$. Since
$\mathcal{O}=\mathcal{O}_{S,s}$ is uniserial, either $a\mid b$ or
$b\mid a$ in $\mathcal{O}$. Changing $x$ to $x^{-1}$, we may therefore
assume that $x$ belongs to $\mathcal{O}$. Since it then maps to
$1$ in $\mathcal{O}_{p}$, where $p=\ker(\mathcal{O}\rightarrow A)$
is the prime corresponding to $s'$, there is an $f\in\mathcal{O}\setminus p$
such that $f(x-1)=0$ in $\mathcal{O}$. But $\mathcal{O}$ being
uniserial has a unique minimal prime $p_{0}$, which must be the nilradical
of $\mathcal{O}$. Since $f\notin p_{0}$, we find that $x-1\in p_{0}$,
i.e.~$x=1+\epsilon$ for some nilpotent $\epsilon$ of $\mathcal{O}$;
thus indeed $x\in\mathcal{O}^{\times}$.

We now consider the following commutative diagram:
\[
\xyC{2pc}\xymatrix{1+p\ar[d]\ar@{^{(}->}[r] & \mathcal{O}^{\times}\ar@{->>}[r]\ar[d] & \left(\mathcal{O}/p\right)^{\times}\ar@{^{(}->}[d]\ar@{^{(}->}[r] & A^{\times}\ar@{^{(}->}[d]\\
1+p\mathcal{O}_{p}\ar@{^{(}->}[r] & \mathcal{O}_{p}^{\times}\ar@{->>}[r] & \left(\mathcal{O}_{p}/p\mathcal{O}_{p}\right)^{\times}\ar@{^{(}->}[r] & K^{\times}
}
\]
The last square is cartesian: $\mathcal{O}/p$ is a uniserial domain,
whence a valuation ring, with fraction field $k(p)=\mathcal{O}_{p}/p\mathcal{O}_{p}$;
it is dominated by the local ring $A\cap k(p)$, and thus equals it;
so $(\mathcal{O}/p)^{\times}=(A\cap k(p))^{\times}=A^{\times}\cap k(p)^{\times}$
in $K^{\times}$. It follows that 
\[
\mathrm{coker}\left(\left(\mathcal{O}/p\right)^{\times}\hookrightarrow(\mathcal{O}_{p}/p\mathcal{O}_{p})^{\times}\right)\rightarrow\mathrm{coker}\left(A^{\times}\hookrightarrow K^{\times}\right)
\]
is injective. On the other hand, the first two squares identify the
kernel of 
\[
\mathrm{coker}\left(\mathcal{O}^{\times}\hookrightarrow\mathcal{O}_{p}^{\times}\right)\rightarrow\mathrm{coker}\left(\left(\mathcal{O}/p\right)^{\times}\hookrightarrow(\mathcal{O}_{p}/p\mathcal{O}_{p})^{\times}\right)
\]
with the cokernel of $1+p\rightarrow1+p\mathcal{O}_{p}$. So it remains
to establish that the latter is trivial, i.e.~that $\mathcal{O}\rightarrow\mathcal{O}_{p}$
induces a surjection $p\rightarrow p\mathcal{O}_{p}$. This is clear:
take $x\in p\mathcal{O}_{p}$ and write it as $x=a/b$ in $\mathcal{O}_{p}$
with $a\in p$ and $b\in\mathcal{O}\setminus p$; if $\mathcal{O}b\subset\mathcal{O}a$,
then $b\in p$, a contradiction; so $\mathcal{O}a\subset\mathcal{O}b$,
i.e. $a=bc$ for some $c\in\mathcal{O}$, with $c\in p$ since $bc=a\in p$
with $b\notin p$; so $x=bc/b=c/1$ in $\mathcal{O}_{p}$.
\end{proof}
\begin{example}
Since Prüfer domains are coherent rings whose local rings are valuation
domains, the proposition applies to their spectra, and more generally
to any scheme with an open cover by spectra of Prüfer domains. 
\end{example}

\section{Functoriality and Convexity}

\subsection{The morphism $\theta$\protect\label{subsec:TheMorphTheta}}

Recall that $\mathcal{L}$ is a subsheaf of $\pi^{\ast}\mathcal{M}_{S}$.
Thus for any morphism $f:T\rightarrow S$ and $(a,\alpha)\in\mathbb{G}_{m,S}^{lft}(T)$,
$a^{\ast}(\mathcal{L}\hookrightarrow\pi^{\ast}\mathcal{M}_{S})$ maps
$\alpha$ to a section $\theta(a,\alpha)$ of $a^{\ast}\pi^{\ast}\mathcal{M}_{S}=f^{\ast}\mathcal{M}_{S}$
over $T$. For $(a,\alpha)$ and $(b,\beta)\in\mathbb{G}_{m,S}^{lft}(T)$,
we find that 
\[
\theta(a,\alpha)\theta(b,\beta)=\theta(a+b,\alpha\cdot\beta)\quad\text{and}\quad\theta(0,1)=1.
\]
This construction is plainly functorial in $f$. We thus obtain a
morphism 
\[
\theta:\mathbb{G}_{m,S}^{lft}\rightarrow\Res_{\eta/S}(\mathbb{G}_{m,\eta})
\]
of Zariski sheaves of groups on $\Sch_{S}$, where $\Res_{\eta/S}(\mathbb{G}_{m,\eta})$
is the sheaf of invertible elements in the Zariski sheaf $\Res_{\eta/S}(\mathbb{G}_{a,\eta})$
of $\mathcal{O}_{S}$-algebras which takes $f:T\rightarrow S$ to
the $\Gamma(T,\mathcal{O}_{T})$-algebra $\Gamma(T,f^{\ast}\mathcal{M}_{S})$.
The restriction of $\theta$ to the neutral component $\mathbb{G}_{m,S}$
of $\mathbb{G}_{m,S}^{lft}$ is the morphism $\mathbb{G}_{m,S}\rightarrow\Res_{\eta/S}(\mathbb{G}_{m,\eta})$
which takes $f$ to 
\[
\mathbb{G}_{m,S}(T)=\Gamma(T,\mathcal{O}_{T})^{\times}\rightarrow\Gamma(T,f^{\ast}\mathcal{M}_{S})^{\times}=\Res_{\eta/S}(\mathbb{G}_{m,\eta})(T).
\]

\begin{rem}
\label{rem:M_SIsRes}Suppose that our sheaf of $\mathcal{O}_{S}$-algebras
$\mathcal{M}_{S}$ is quasi-coherent and set $\eta=\Spec(\mathcal{M}_{S})$,
with affine structure morphism $\iota:\eta\rightarrow S$. For any
$f:T\rightarrow S$ as above, we may now form the following cartesian
diagram of schemes:
\[
\xymatrix{T_{\eta}\ar[r]^{f'}\ar[d]_{\iota'} & \eta\ar[d]^{\iota}\\
T\ar[r]^{f} & S
}
\]
Since $\iota$ is affine, $f^{\ast}\mathcal{M}_{S}=f^{\ast}\iota_{\ast}\mathcal{O}_{\eta}=\iota_{\ast}^{\prime}f^{\prime\ast}\mathcal{O}_{\eta}=\iota_{\ast}^{\prime}\mathcal{O}_{T_{\eta}}$.
Thus 
\[
\Res_{\eta/S}(\mathbb{G}_{a,\eta})(T)=\mathbb{G}_{a}(T_{\eta})\qquad\text{and}\qquad\Res_{\eta/S}(\mathbb{G}_{m,\eta})(T)=\mathbb{G}_{m}(T_{\eta}).
\]
In other words, $\Res_{\eta/S}(\mathbb{G}_{a,\eta})$ and $\Res_{\eta/S}(\mathbb{G}_{m,\eta})$
are then respectively the Weil restriction sheaves of $\mathbb{G}_{a,\eta}$
and $\mathbb{G}_{m,\eta}$ along $\iota:\eta\rightarrow S$, as suggested
by the notation. 
\end{rem}

\subsection{Convexity, I}

Let again $f:T\rightarrow S$ be a morphism of schemes. 

\subsubsection{~}

We can then form the following commutative diagram of exact sequences
of sheaves of groups on the small Zariski site $T_{Zar}$ of $T$:
\[
\xyC{1pc}\xymatrix{1\ar[r] & f^{-1}\mathcal{O}_{S}^{\times}\ar[r]\ar[d] & f^{-1}\mathcal{M}_{S}^{\times}\ar[d]\ar[r] & f^{-1}\mathcal{D}_{S}\ar@{=}[d]\ar[r] & 1\\
1\ar[r] & \mathcal{O}_{T}^{\times}\ar[r]\ar@{=}[d] & \mathcal{O}_{T}^{\times}\coprod_{f^{-1}\mathcal{O}_{S}^{\times}}f^{-1}\mathcal{M}_{S}^{\times}\ar[r]\ar[d]_{\vartheta} & f^{-1}\mathcal{D}_{S}\ar[r]\ar[d] & 1\\
1\ar@{..>}[r] & \mathcal{O}_{T}^{\times}\ar[r]\ar@{=}[d] & f^{\ast}\mathcal{M}_{S}^{\times}\ar@{..>}[d]\ar[r] & f^{\ast}\mathcal{M}_{S}^{\times}/\mathrm{im}\left(\mathcal{O}_{T}^{\times}\right)\ar@{..>}[d]\ar[r] & 1\\
1\ar[r] & \mathcal{O}_{T}^{\times}\ar[r] & \mathcal{M}_{T}^{\times}\ar[r] & \mathcal{D}_{T}\ar[r] & 1
}
\]
For the first line, we apply the exact functor $f^{-1}:\Shv(S_{Zar})\rightarrow\Shv(T_{Zar})$
to the exact sequence defining $\mathcal{D}_{S}=\mathcal{M}_{S}^{\times}/\mathcal{O}_{S}^{\times}$;
note here that $f^{-1}(\mathcal{O}_{S}^{\times}\hookrightarrow\mathcal{O}_{S})$
and $f^{-1}(\mathcal{M}_{S}^{\times}\hookrightarrow\mathcal{M}_{S})$
respectively identify $f^{-1}(\mathcal{O}_{S}^{\times})$ with $(f^{-1}\mathcal{O}_{S})^{\times}\hookrightarrow f^{-1}\mathcal{O}_{S}$
and $f^{-1}(\mathcal{M}_{S}^{\times})$ with $(f^{-1}\mathcal{M}_{S})^{\times}\hookrightarrow f^{-1}\mathcal{M}_{S}$,
justifying our ambiguous notations. The second line is obtained by
pushout along the morphism $f^{-1}\mathcal{O}_{S}^{\times}\rightarrow\mathcal{O}_{T}^{\times}$
induced by the structure morphism $f^{-1}\mathcal{O}_{S}\rightarrow\mathcal{O}_{T}$.
The morphism to the third line is induced by the $f^{-1}\mathcal{O}_{S}$-algebra
morphisms from $\mathcal{O}_{T}$ and $f^{-1}\mathcal{M}_{S}$ to
their tensor product 
\[
f^{\ast}\mathcal{M}_{S}=\mathcal{O}_{T}\otimes_{f^{-1}\mathcal{O}_{S}}f^{-1}\mathcal{M}_{S}.
\]
The dotted arrows to the fourth line only exist when $f^{-1}\mathcal{O}_{S}\rightarrow\mathcal{O}_{T}$
extends to $f^{-1}\mathcal{M}_{S}\rightarrow\mathcal{M}_{T}$, inducing
a morphism of $\mathcal{O}_{T}$-algebras $f^{\ast}\mathcal{M}_{S}\rightarrow\mathcal{M}_{T}$,
which restricts to $f^{\ast}\mathcal{M}_{S}^{\times}\rightarrow\mathcal{M}_{T}^{\times}$
-- it forces injectivity of $\mathcal{O}_{T}\rightarrow f^{\ast}\mathcal{M}_{S}$
and $\mathcal{O}_{T}^{\times}\rightarrow f^{\ast}\mathcal{M}_{S}^{\times}$.
This holds when $f$ is flat: the structure morphism $\mathcal{O}_{S}\rightarrow f_{\ast}\mathcal{O}_{T}$
then maps $\mathcal{O}_{S}^{reg}$ into $f_{\ast}\mathcal{O}_{T}^{reg}$,
giving a morphism $\mathcal{M}_{S}\rightarrow f_{\ast}\mathcal{M}_{T}$
whose adjoint is the desired extension.

\subsubsection{~}

We claim that the second line of our diagram may be identified with
the restriction of the connected-étale short exact sequence of $\mathbb{G}_{m,S}^{lft}$,
from the big Zariski site of $S$ to the small Zariski site of $T$,
namely the short exact sequence 
\begin{equation}
\xymatrix{1\ar[r] & \mathbb{G}_{m,S}\vert_{T_{Zar}}\ar[r] & \mathbb{G}_{m,S}^{lft}\vert_{T_{Zar}}\ar[r] & \Pi_{S}\vert_{T_{Zar}}\ar[r] & 1}
.\label{eq:RestOfConnecEtalSeq2TZar}
\end{equation}
Since $\mathbb{G}_{m,S}\vert_{T_{Zar}}=\mathcal{O}_{T}^{\times}$,
we just have to extend $f^{-1}\mathcal{O}_{S}^{\times}\rightarrow\mathcal{O}_{T}^{\times}$
to a morphism $f^{-1}\mathcal{M}_{S}^{\times}\rightarrow\mathbb{G}_{m,S}^{lft}\vert_{T_{Zar}}$;
by adjunction, this amounts to extending $\mathcal{O}_{S}^{\times}\rightarrow f_{\ast}\mathcal{O}_{T}^{\times}$
to a morphism $\mathcal{M}_{S}^{\times}\rightarrow f_{\ast}(\mathbb{G}_{m,S}^{lft}\vert_{T_{Zar}})$.
The latter is given on an open $U$ of $S$ by 
\[
\mathcal{M}_{S}^{\times}(U)\ni x\mapsto(a,\alpha)\in\mathbb{G}_{m,S}^{lft}(f^{-1}(U))
\]
where $a:f^{-1}(U)\rightarrow\Pi_{S}$ is the composition of $f^{-1}(U)\rightarrow U$
with the section $d:U\rightarrow\Pi_{S}$ corresponding to the image
$d=\mathrm{div}(x)$ of $x$ in $\mathcal{D}_{S}(U)=\Hom_{S}(U,\Pi_{S})$,
and $\alpha=f^{\ast}(x)$, which is indeed a generating section of
$a^{\ast}\mathcal{L}=f^{\ast}(\mathcal{O}_{U}x)$ over $f^{-1}(U)$. 

The resulting identification of $\Pi_{S}\vert_{T_{Zar}}$ with $f^{-1}\mathcal{D}_{S}$
is well-known: the espace étalé of the pull-back sheaf $f^{-1}\mathcal{D}_{S}$
equals $\Pi_{S}\times_{S}T$. Similarly, one checks that the restriction
of the morphism $\theta:\mathbb{G}_{m,S}^{lft}\rightarrow\Res_{\eta/S}(\mathbb{G}_{m,\eta})$
from section~\ref{subsec:TheMorphTheta} to $T_{Zar}$ becomes the
middle multiplication morphism $\vartheta:\mathcal{O}_{T}^{\times}\coprod_{f^{-1}\mathcal{O}_{S}^{\times}}f^{-1}\mathcal{M}_{S}^{\times}\rightarrow f^{\ast}\mathcal{M}_{S}^{\times}$.

\subsubsection{The flat case}

Suppose that $f:T\rightarrow S$ is flat. We may then retrieve the
exact sequence (\ref{eq:RestOfConnecEtalSeq2TZar}) by pull-back of
the third or fourth line of our bigger diagram under the appropriate
morphism from $\Pi_{S}\vert_{T_{Zar}}$ to the relevant last term:
\[
\xyC{1pc}\xymatrix{1\ar[r] & \mathbb{G}_{m,S}\vert_{T_{Zar}}\ar[r]\ar@{=}[d] & \mathbb{G}_{m,S}^{lft}\vert_{T_{Zar}}\ar[r]\ar[d]_{\vartheta} & \Pi_{S}\vert_{T_{Zar}}=f^{-1}\mathcal{D}_{S}\ar[r]\ar[d] & 1\\
1\ar[r] & \mathcal{O}_{T}^{\times}\ar[r]\ar@{=}[d] & f^{\ast}\mathcal{M}_{S}^{\times}\ar[d]\ar[r] & f^{\ast}\mathcal{M}_{S}^{\times}/\mathcal{O}_{T}^{\times}\ar[d]\ar[r] & 1\\
1\ar[r] & \mathcal{O}_{T}^{\times}\ar[r] & \mathcal{M}_{T}^{\times}\ar[r] & \mathcal{D}_{T}\ar[r] & 1
}
\]
Moreover, all vertical maps in this diagram are injective. Indeed,
the stalks of the last column morphisms at $t\in T$ above $s\in S$
are given by
\[
\mathcal{D}_{S,s}=\mathcal{M}_{S,s}^{\times}/\mathcal{O}_{S,s}^{\times}\rightarrow\left(\mathcal{O}_{T,t}\otimes_{\mathcal{O}_{S,s}}\mathcal{M}_{S,s}\right)^{\times}/\mathcal{O}_{T,t}^{\times}\rightarrow\mathcal{M}_{T,t}^{\times}/\mathcal{O}_{T,t}^{\times}=\mathcal{D}_{T,t}.
\]
So the claimed injectivity derives from the following lemma, which
also shows that the poset structure on $\Pi_{S}\vert_{T_{Zar}}$ is
induced by the poset structure on $\mathcal{D}_{T}$. 
\begin{lem}
With notations as above, we have embeddings 
\[
\mathcal{M}_{S,s}\hookrightarrow\mathcal{O}_{T,t}\otimes_{\mathcal{O}_{S,s}}\mathcal{M}_{S,s}\hookrightarrow\mathcal{M}_{T,t}.
\]
Moreover, taking intersections in $\mathcal{O}_{T,t}\otimes_{\mathcal{O}_{S,s}}\mathcal{M}_{S,s}$
or $\mathcal{M}_{T,t}$, we have 
\[
\mathcal{O}_{S,s}=\mathcal{M}_{S,s}\cap\mathcal{O}_{T,t},\quad\mathcal{O}_{S,s}^{reg}=\mathcal{M}_{S,s}^{\times}\cap\mathcal{O}_{T,t}=\mathcal{M}_{S,s}^{\times}\cap\mathcal{O}_{T,t}^{reg},\quad\mathcal{O}_{S,s}^{\times}=\mathcal{M}_{S,s}^{\times}\cap\mathcal{O}_{T,t}^{\times}.
\]
\end{lem}

\begin{proof}
First note that $\mathcal{O}_{T,t}\otimes_{\mathcal{O}_{S,s}}\mathcal{M}_{S,s}\rightarrow\mathcal{M}_{T,t}$
is injective since it is a localization of $\mathcal{O}_{T,t}\rightarrow\mathcal{M}_{T,t}$,
which is injective. Note also that $\mathcal{O}_{S,s}\rightarrow\mathcal{O}_{T,t}$
is flat and local, hence faithfully flat, thus universally injective.
Write $x\in\mathcal{M}_{S,s}$ as $x=a/b$ with $a\in\mathcal{O}_{S,s}$
and $b\in\mathcal{O}_{S,s}^{reg}$. If $x$ maps to $0$ in $\mathcal{M}_{T,t}$,
then so does $a$ in $\mathcal{O}_{T,t}$, so $a=0$ in $\mathcal{O}_{S,s}$
and $x=0$ in $\mathcal{M}_{S,s}$. Thus $\mathcal{M}_{S,s}\hookrightarrow\mathcal{M}_{T,t}$
is injective, and so is $\mathcal{M}_{S,s}\hookrightarrow\mathcal{O}_{T,t}\otimes_{\mathcal{O}_{S,s}}\mathcal{M}_{S,s}$.
If $x$ maps to $\mathcal{O}_{T,t}\subset\mathcal{M}_{T,t}$, then
$\mathcal{O}_{T,t}a\subset\mathcal{O}_{T,t}b$ in $\mathcal{O}_{T,t}$,
whence $\mathcal{O}_{S,s}a\subset\mathcal{O}_{S,s}b$ in $\mathcal{O}_{S,s}$
since $I\mathcal{O}_{T,t}\cap\mathcal{O}_{S,s}=I$ for any ideal $I$
of $\mathcal{O}_{S,s}$. Thus $\mathcal{O}_{S,s}=\mathcal{M}_{S,s}\cap\mathcal{O}_{T,t}$,
say in $\mathcal{M}_{T,t}$. Intersecting this with $\mathcal{M}_{S,s}^{\times}$
gives $\mathcal{O}_{S,s}^{reg}=\mathcal{M}_{S,s}^{\times}\cap\mathcal{O}_{T,t}$,
while taking invertible elements gives $\mathcal{O}_{S,s}^{\times}=\mathcal{M}_{S,s}^{\times}\cap\mathcal{O}_{T,t}^{\times}$.
Since $\mathcal{O}_{S,s}^{reg}\subset\mathcal{O}_{T,t}^{reg}$ and
$\mathcal{O}_{S,s}^{reg}=\mathcal{M}_{S,s}^{\times}\cap\mathcal{O}_{T,t}$,
also $\mathcal{O}_{S,s}^{reg}=\mathcal{M}_{S,s}^{\times}\cap\mathcal{O}_{T,t}^{reg}$. 
\end{proof}

\subsubsection{Convex Morphisms}

In the previous diagram, the second line is some sort of convex envelope
of the first line in the third line. Here is an actual local statement: 
\begin{lem}
\label{lem:ConvexEnveloppe}For $t\in T$ above $s\in S$, the middle
group in 
\[
\mathcal{D}_{S,s}\subset(\mathcal{O}_{T,t}\otimes_{\mathcal{O}_{S,s}}\mathcal{M}_{S,s})^{\times}/\mathcal{O}_{T,t}^{\times}\subset\mathcal{D}_{T,t}
\]
is the convex envelope of the subgroup $\mathcal{D}_{S,s}$ in the
partially ordered group $\mathcal{D}_{T,t}$. 
\end{lem}

\begin{proof}
Let $\mathcal{D}$ be the middle group. Lift $\beta\in\mathcal{D}$
to $b\in(\mathcal{O}_{T,t}\otimes_{\mathcal{O}_{S,s}}\mathcal{M}_{S,s})^{\times}$
and write $b=x/d$, $b^{-1}=y/c$ with $x,y\in\mathcal{O}_{T,t}$
and $c,d\in\mathcal{O}_{S,s}^{reg}$; then $x=bd$ and $y=cb^{-1}$
belong to $\mathcal{O}_{T,t}\cap\mathcal{M}_{T,t}^{\times}=\mathcal{O}_{T,t}^{reg}$,
$a=1/d$ and $c$ belong to $\mathcal{M}_{S,s}^{\times}$, with $ax=b$
and $by=c$ in $\mathcal{M}_{T,t}$; thus $\alpha\leq\beta\leq\gamma$
in $\mathcal{D}_{T,t}$, where $\alpha,\gamma\in\mathcal{D}_{S,s}$
are the images of $a,c\in\mathcal{M}_{S,s}^{\times}$. Suppose conversely
that $\alpha\leq\beta\leq\gamma$ in $\mathcal{D}_{T,t}$ with $\alpha,\gamma\in\mathcal{D}_{S,s}$,
and lift them to $a,c\in\mathcal{M}_{S,s}^{\times}$ and $b\in\mathcal{M}_{T,t}^{\times}$,
so that $ax=b$ and $by=c$ in $\mathcal{M}_{T,t}$ for some $x,y\in\mathcal{O}_{T,t}^{reg}$;
then $xy=c/a\in\mathcal{M}_{S,s}^{\times}$ is invertible in $\mathcal{O}_{T,t}\otimes\mathcal{M}_{S,s}$,
thus so are $x,y\in\mathcal{O}_{T,t}$, along with $b=ax$, i.e.~$\beta$
belongs to the middle group $\mathcal{D}$.
\end{proof}
\begin{prop}
\label{prop:ConvexNP}The following conditions on the flat $f:T\rightarrow S$
are equivalent:
\begin{enumerate}
\item The restriction $\vartheta$ of $\theta$ to $T_{Zar}$ is an isomorphism.
\item $f^{-1}\mathcal{D}_{S}\hookrightarrow f^{\ast}\mathcal{M}_{S}^{\times}/\mathcal{O}_{T}^{\times}$
is an isomorphism.
\item For $t\in T$ above $s\in S$, the image of $\mathcal{D}_{S,s}\hookrightarrow\mathcal{D}_{T,t}$
is convex.
\end{enumerate}
\end{prop}

\begin{proof}
This follows from the previous discussion. 
\end{proof}
\begin{defn}
We say that a flat $f$ is \emph{convex} if it satisfies the above
conditions. 
\end{defn}

\subsection{Functoriality}

Still assuming that $f:T\rightarrow S$ is flat, we may promote the
commutative diagram of the previous subsection to a commutative diagram
\[
\xyC{1pc}\xymatrix{1\ar[r] & \mathbb{G}_{m,S}\times_{S}T\ar[r]\ar@{=}[d] & \mathbb{G}_{m,S}^{lft}\times_{S}T\ar[r]\ar[d] & \Pi_{S}\times_{S}T\ar[r]\ar[d] & 1\\
1\ar[r] & \mathbb{G}_{m,T}\ar[r]\ar@{=}[d] & \mathbb{G}_{m,f}^{lft}\ar[r]\ar[d] & \Pi_{f}\ar[r]\ar[d] & 1\\
1\ar[r] & \mathbb{G}_{m,T}\ar[r] & \mathbb{G}_{m,T}^{lft}\ar[r] & \Pi_{T}\ar[r] & 1
}
\]
of smooth commutative group schemes over $T$ as follows. For the
right column, we take the étale group schemes associated to the Zariski
sheaves 
\[
\xymatrix{f^{-1}(\mathcal{D}_{S})=f^{-1}(\mathcal{M}_{S}^{\times}/\mathcal{O}_{S}^{\times})\ar[r] & f^{\ast}\mathcal{M}_{S}^{\times}/\mathcal{O}_{T}^{\times}\ar[r] & \mathcal{M}_{T}^{\times}/\mathcal{O}_{T}^{\times}=\mathcal{D}_{T}}
.
\]
For the second line, we take the pullback of the third by $\Pi_{f}\rightarrow\Pi_{T}$.
For the morphism from the first to the second line, we thus just need
to define a morphism 
\[
f^{\star}:\mathbb{G}_{m,S}^{lft}\times_{S}T\rightarrow\mathbb{G}_{m,T}^{lft}
\]
extending $\mathbb{G}_{m,S}\times_{S}T\simeq\mathbb{G}_{m,T}$ and
inducing the already constructed $T$-morphism $f^{\star}:\Pi_{S}\times_{S}T\rightarrow\Pi_{T}$.
Let us write $\pi_{S}:\Pi_{S}\rightarrow S$ and $\pi_{T}:\Pi_{T}\rightarrow T$
for the structural morphisms, $\mathcal{L}_{S}\subset\pi_{S}^{\ast}\mathcal{M}_{S}$
and $\mathcal{L}_{T}\subset\pi_{T}^{\ast}\mathcal{M}_{T}$ for the
canonical primitive line bundles on $\Pi_{S}$ and $\Pi_{T}$. They
respectively pull back to the line bundles 
\[
p_{1}^{\ast}\mathcal{L}_{S}\subset p_{1}^{\ast}\pi_{S}^{\ast}\mathcal{M}_{S}=p_{2}^{\ast}f^{\ast}\mathcal{M}_{S}\quad\text{and}\quad(f^{\star})^{\ast}\mathcal{L}_{T}\subset(f^{\star})^{\ast}\circ\pi_{T}^{\ast}\mathcal{M}_{T}=p_{2}^{\ast}\mathcal{M}_{T}
\]
on $\Pi_{S}\times_{S}T$, where $p_{1}:\Pi_{S}\times_{S}T\rightarrow\Pi_{S}$
and $p_{2}:\Pi_{S}\times_{S}T\rightarrow T$ are the projections.
We have seen that $f^{\ast}\mathcal{O}_{S}=\mathcal{O}_{T}$ extends
to an embedding $f^{\ast}\mathcal{M}_{S}\hookrightarrow\mathcal{M}_{T}$,
which pulls back to an embedding $p_{2}^{\ast}f^{\ast}\mathcal{M}_{S}\hookrightarrow p_{2}^{\ast}\mathcal{M}_{T}$.
We claim that the latter identifies $p_{1}^{\ast}\mathcal{L}_{S}$
with $(f^{\star})^{\ast}\mathcal{L}_{T}$. This can be checked locally,
so let $x$ be a point of $\Pi_{S}\times_{S}T$, i.e.~$x=(\overline{d},t)$
mapping to $t\in T$ above $s\in S$, with $\overline{d}\in\mathcal{D}_{S,s}$
represented by $d\in\mathcal{M}_{S,s}^{\times}$, so that $f^{\star}$
maps $x$ to the point corresponding to $f^{\ast}\overline{d}\in\mathcal{D}_{T,t}$
represented by $f^{\ast}d\in\mathcal{M}_{T,t}^{\times}$ in the fiber
of $\pi_{T}:\Pi_{T}\rightarrow T$ above $t$. The stalk of $p_{2}^{\ast}f^{\ast}\mathcal{M}_{S}\hookrightarrow p_{2}^{\ast}\mathcal{M}_{T}$
at $x$ is then the morphism $\mathcal{O}_{T,t}\otimes_{\mathcal{O}_{S,s}}\mathcal{M}_{S,s}\hookrightarrow\mathcal{M}_{T,t}$,
which indeed maps the local generator $1\otimes d$ of $p_{1}^{\ast}\mathcal{L}_{S}$
to the local generator $f^{\ast}d$ of $(f^{\star})^{\ast}\mathcal{L}_{T}$.
We may now define our morphism $f^{\star}$ as follows: for a $T$-scheme
$U$ and $(a,\alpha)\in\mathbb{G}_{m,S}^{lft}(U)$, we set 
\[
f^{\star}(a,\alpha)=(b,\beta)
\]
 where $b:U\rightarrow\Pi_{T}$ is the precomposition of $f^{\star}:\Pi_{S}\times_{S}T\rightarrow\Pi_{T}$
with the $T$-lift $a':U\rightarrow\Pi_{S}\times_{S}T$ of the $S$-morphism
$a:U\rightarrow\Pi_{S}$, and $\beta\in\Gamma(U,b^{\ast}\mathcal{L}_{T})$
is the generating section corresponding to $\alpha\in\Gamma(U,a^{\ast}\mathcal{L}_{S})$
under the isomorphism 
\[
a^{\ast}\mathcal{L}_{S}=a^{\prime*}\circ p_{1}^{\ast}\mathcal{L}_{S}\simeq a^{\prime*}(f^{\star})^{\ast}\mathcal{L}_{T}=b^{\ast}\mathcal{L}_{T}.
\]
As before, all vertical maps in our diagram are monomorphisms. Since
the last column is made of étale schemes, all of our vertical maps
are also étale, and thus open immersions. The first two lines are
equal if and only if $f$ is convex. Moreover: 
\begin{prop}
The following conditions on the flat $f:T\rightarrow S$ are equivalent: 
\begin{enumerate}
\item $f^{\star}:\mathbb{G}_{m,S}^{lft}\times_{S}T\hookrightarrow\mathbb{G}_{m,T}^{lft}$
is an isomorphism.
\item $\Pi_{S}\times_{S}T\hookrightarrow\Pi_{T}$ is an isomorphism.
\item For every $t\in T$ above $s\in S$, $\mathcal{D}_{S,s}\hookrightarrow\mathcal{D}_{T,t}$
is an isomorphism. 
\end{enumerate}
\end{prop}

\begin{example}
Plainly, $\mathbb{G}_{m,S}^{lft}\times_{S}U\simeq\mathbb{G}_{m,U}^{lft}$
for any open $U$ of $S$. For $s\in S$ and with $S(s)=\Spec(\mathcal{O}_{S,s})$,
we find that $\mathbb{G}_{m,S}^{lft}\times_{S}S(s)\simeq\mathbb{G}_{m,S(s)}^{lft}$
if and only if 
\[
\forall t\in S(s):\qquad\mathcal{O}_{S,t}^{reg}=\mathcal{O}_{S(s),t}^{reg}\quad\text{in}\quad\mathcal{O}_{S,t}=\mathcal{O}_{S(s),t}.
\]
This holds true if $\WeakAss(S)$ is locally finite at $s$ by proposition~\ref{prop:NoPathologies}. 
\end{example}

\subsection{Convexity, II}

We have defined a flat morphism $f:T\rightarrow S$ to be convex if
for every $t\in T$ mapping to $s\in S$, the embedding $\mathcal{D}_{S,s}\hookrightarrow\mathcal{D}_{T,t}$
realizes $\mathcal{D}_{S,s}$ has a convex subgroup of $\mathcal{D}_{T,t}$.
We will investigate a slightly stronger, but purely local condition.
Namely, consider the commutative diagram 
\[
\xymatrix{\mathcal{D}_{S,s}\ar[r]\ar[d] & \mathcal{D}_{T,t}\ar[d]\\
\mathcal{D}(\mathcal{O}_{S,s})\ar[r] & \mathcal{D}(\mathcal{O}_{T,t})
}
\]
Since both vertical maps are convex embeddings, we find that
\[
\mathcal{D}(\mathcal{O}_{S,s})\text{ convex in }\mathcal{D}(\mathcal{O}_{T,t})\Rightarrow\mathcal{D}_{S,s}\text{ convex in }\mathcal{D}(\mathcal{O}_{T,t})\Rightarrow\mathcal{D}_{S,s}\text{ convex in }\mathcal{D}_{T,t}.
\]
This suggests the following definition. For a commutative ring $A$,
we denote by $\mathcal{D}(A)=\mathcal{Q}(A)^{\times}/A^{\times}$
its \emph{group of divisibility}, a partially ordered commutative
group (pocg) with positive cone $\mathcal{D}^{+}(A)=\Reg(A)/A^{\times}$.
A flat morphism $\varphi:A\rightarrow B$ induces a pocg morphism
$\mathcal{D}(\varphi):\mathcal{D}(A)\rightarrow\mathcal{D}(B)$. If
$\varphi$ is faithfully flat, $\mathcal{D}(\varphi)$ is injective
and $\mathcal{D}^{+}(A)=\mathcal{D}(\varphi)^{-1}(\mathcal{D}^{+}(B))$,
i.e.~$\mathcal{D}(\varphi)$ is a strict embedding of pocgs. 
\begin{defn}
A faithfully flat ring morphism $\varphi:A\rightarrow B$ is \emph{naively
convex} if and only if it satisfies the following equivalent conditions: 
\begin{enumerate}
\item It realizes $\mathcal{D}(A)$ as a convex subgroup of $\mathcal{D}(B)$.
\item Any unit of $B\otimes\mathcal{Q}(A)$ is the product of a unit of
$B$ and a unit of $\mathcal{Q}(A)$.
\item Any divisor in $B$ of a regular element of $A$ is the product of
a unit of $B$ and a regular element of $A$.
\end{enumerate}
The proof of the equivalence of the three conditions is easy and left
to the reader. Note that our faithfully flat $\varphi:A\rightarrow B$
is universally injective. Thus for any ideal $I$ of $A$, we have
$A/I\hookrightarrow B/IB$, i.e.~$I=A\cap IB$ in $B$. In particular
for $a_{1},a_{2}\in A$, if $a_{1}\mid a_{2}$ in $B$, then also
$a_{1}\mid a_{2}$ in $A$. It follows that if $\varphi$ is naively
convex, it preserves irreducible regular elements. Here is a partial
converse:
\end{defn}

\begin{lem}
Suppose that $A$ and $B$ are unique factorization domains. Then
$\varphi$ is naively convex if and only if it maps prime elements
of $A$ to prime elements of $B$.
\end{lem}

\begin{proof}
This is particularly clear on the first characterization of naively
convex maps: $\mathcal{D}(A)=\mathcal{Q}(A)^{\times}/A^{\times}$
is now the free abelian group on the set of equivalence classes $[p]=A^{\times}p$
of primes $p$ of $A$, likewise for $\mathcal{D}(B)=\mathcal{Q}(B)^{\times}/B^{\times}$,
and $\mathcal{D}(\varphi)$ maps $[p]$ to $\sum n_{q}[q]$ where
$p=u\prod q^{n_{q}}$ is the prime factorization of $p$ in $B$. 
\end{proof}
Suppose that $\varphi$ is naively convex. Then for any divisor $b$
in $B$ of a regular element of $A$, $(1)$ there is a smallest ideal
$I$ of $A$ such that $b\in IB$, and $(2)$ there is a smallest
principal ideal $J$ of $A$ such that $I\subset J$. Indeed we know
that $b=ua$ with $u\in B^{\times}$ and $a\in\Reg(A)$, so we may
take $I=J=Aa$. These properties are respectively implied by $(1')$
$\varphi$ is \emph{Ohm-Rush}, i.e.~for \emph{any }element $b$ of
$B$, there is a smallest ideal $I$ of $A$ such that $b\in IB$
-- in which case $I$ is necessarily finitely generated, and $(2')$
$A$ is a GCD-domain -- in which case for \emph{any }finitely generated
ideal $I$ of $A$, there is a smallest principal ideal $J$ of $A$
such that $I\subset J$. In fact:
\begin{lem}
If $A$ is a GCD-domain, $\varphi$ is Ohm-Rush, and $\Spec(\varphi)$
has integral fibers over $\{p:\text{p is minimal over some }a\in p,a\neq0\}$,
then $\varphi$ is naively convex.
\end{lem}

\begin{proof}
Suppose that $a=b_{1}b_{2}$ in $B$ with $b_{i}\in B$ and $a\in\Reg(A)$.
Let $I_{i}\subset A$ be the (finitely generated) smallest ideal of
$A$ such that $b_{i}\in I_{i}B$ and let $J_{i}=a_{i}A$ be the smallest
principal ideal of $A$ such that $I_{i}\subset J_{i}$. Thus $b_{i}\in I_{i}B\subset J_{i}B=a_{i}B$,
whence $b_{i}=a_{i}b_{i}^{\prime}$ for some $b_{i}^{\prime}\in B$.
In particular $a_{1}a_{2}b_{1}^{\prime}b_{2}^{\prime}=a$ in $B$,
so $a_{1}a_{2}\mid a$ in $B$ and $A$, say $a=a_{1}a_{2}a'$ for
some $a'\in A$. Since $a$ is regular in $A$, so are $a_{1}$, $a_{2}$
and $a'$, in $A$ and $B$. Since $a_{1}a_{2}b_{1}^{\prime}b_{2}^{\prime}=a_{1}a_{2}a'$
in $B$, $b_{1}^{\prime}b_{2}^{\prime}=a'$ in $B$. Let $I_{i}^{\prime}\subset A$
be the (finitely generated) smallest ideal of $A$ such that $b_{i}^{\prime}\in I_{i}^{\prime}B$
and let $J_{i}^{\prime}$ be the smallest principal ideal of $A$
such that $I_{i}^{\prime}\subset J_{i}^{\prime}$. Since $b_{i}=a_{i}b_{i}^{\prime}\subset a_{i}I_{i}^{\prime}B$,
$I_{i}\subset a_{i}I_{i}^{\prime}\subset a_{i}J_{i}^{\prime}$ by
definition of $I_{i}$, so $a_{i}A=J_{i}\subset a_{i}J_{i}^{\prime}\subset a_{i}A$
by definition of $J_{i}$, whence $J_{i}^{\prime}=A$ since $a_{i}\in\Reg(A)$.
If $a'$ is not invertible in $A$, there is a minimal prime ideal
$p$ of $A$ containing $a'$; by~\cite[Theorem 3.1]{Sh74}, any
such $p$ is a PF-prime, i.e.~stable under taking GCD's; by assumption,
$pB$ is a prime ideal of $B$; since $b_{1}^{\prime}b_{2}^{\prime}=a^{\prime}\in pB$,
$b_{i}^{\prime}\in pB$ for some $i\in\{1,2\}$, and then $I_{i}^{\prime}\subset p$
by definition of $I_{i}^{\prime}$, so $J_{i}^{\prime}\subset p$
by definition of $J_{i}^{\prime}$; this contradicts $J_{i}^{\prime}=A$,
so $a'$ is invertible in $A$. Since $b_{1}^{\prime}b_{2}^{\prime}=a'$
in $B$, also $b_{i}^{\prime}\in B^{\times}$. Thus $b_{i}=a_{i}b_{i}^{\prime}\in\Reg(A)\cdot B^{\times}$,
which proves the lemma. 
\end{proof}

\subsection{Smooth morphisms over Prüfer schemes are convex}

Returning to a flat morphism of schemes $f:T\rightarrow S$, we want
to apply the previous lemma to the (faithfully) flat local morphisms
$\varphi:A\rightarrow B$, with $A=\mathcal{O}_{S,s}$ and $B=\mathcal{O}_{T,t}$
for $t\in T$ above $s\in S$. So we first require that $S$ is a
\emph{GCD-scheme}, i.e.~all local rings of $S$ are GCD-domains.
For instance, $S$ could be \emph{regular}, or \emph{coherent regular},
or \emph{Prüfer}, which respectively imply that its local rings are
noetherian regular (hence UFD's), coherent regular (hence GCD-domains
by \cite[Theorem 5.21]{Va76}), or valuations domains. We next want
$\varphi$ to be Ohm-Rush, which holds true if $f$ is locally of
finite type and $A$ is henselian \cite[\href{https://stacks.math.columbia.edu/tag/05U7}{Tag 05U7}]{SP}
or a valuation domain \cite[\href{https://stacks.math.columbia.edu/tag/0ASX}{Tag 0ASX}]{SP}.
Finally, we want $\Spec(\varphi)$ to have integral fibers over the
PF-primes of $A$, which holds true if all fibers of $f$ are integral,
or if $f$ is smooth and $A$ is a valuation domain; for the latter
case, note that for \emph{any} $p\in\Spec(A)$, $B/pB=\mathcal{O}_{T,t}/p\mathcal{O}_{T,t}$
is a local ring of $T\times_{S}\Spec(A/p)$, hence normal (in particular
integral) since $f$ is smooth and $A/p$ is normal, being yet another
valuation domain. In particular: 
\begin{cor}
\label{cor:SmoothOverPruferIsConvexe}Over a Prüfer scheme $S$, any
smooth morphism is convex.
\end{cor}

\section{The Néron property}

\subsection{The Néron property\protect\label{subsec:NeronProp}}

Suppose that $\WeakAss(S)=\Max(S)$, and is locally finite in $S$.
By proposition~\ref{prop:NoPathologies}, $\mathcal{M}_{S}$ is quasi-coherent
and for $\eta=\Spec(\mathcal{M}_{S})$, we find 
\[
\mathbb{G}_{m,S}^{lft}\times_{S}\eta=\mathbb{G}_{m,\eta}^{lft}=\mathbb{G}_{m,\eta}
\]
where the second equality comes from the triviality of $\mathcal{D}_{\eta,x}$
for all $x\in\eta$. We have also seen in remark~\ref{rem:M_SIsRes}
that $\Res_{\eta/S}(\mathbb{G}_{m,\eta})$ is the Weil restriction.
In this context, 
\[
\theta:\mathbb{G}_{m,S}^{lft}\rightarrow\Res_{\eta/S}(\mathbb{G}_{m,\eta})
\]
evaluated at an $S$-scheme $T$ merely computes the canonical base
change morphism
\[
\mathbb{G}_{m,S}^{lft}(T)=\Hom_{S}(T,\mathbb{G}_{m,S}^{lft})\rightarrow\Hom_{\eta}(T_{\eta},\mathbb{G}_{m,\eta}^{lft})=\mathbb{G}_{m,\eta}(T_{\eta}).
\]
Thus if $T\rightarrow S$ is flat and convex, proposition~\ref{prop:ConvexNP}
implies that it is an isomorphism
\[
\mathbb{G}_{m,S}^{lft}(T)\stackrel{\simeq}{\longrightarrow}\mathbb{G}_{m,\eta}(T_{\eta}).
\]

\subsection{Cocharacters of smooth finitely presented groups}

Let $G$ be a finitely presented smooth group scheme over $S$, and
let $X(G):\Sch_{S}^{\circ}\rightarrow\Grp$ be its (fpqc) sheaf of
characters, which maps an $S$-scheme $T$ to the (abelian) group
\[
X(G)(T)=\Hom_{\Grp/T}(G_{T},\mathbb{G}_{m,T}).
\]

\begin{prop}
The sheaf $X(G)$ satisfies the valuative criterion of properness. 
\end{prop}

\begin{proof}
We have to show that for any $S$-valuation domain $\mathcal{O}$
with fraction field $K$, the natural map $X(G)(\mathcal{O})\rightarrow X(G)(K)$
is bijective. We may assume that $S=\Spec(\mathcal{O})$, in which
case $\eta=\Spec(\mathcal{M}_{S})=\Spec(K)$. Since $T=G$ is smooth
over $S$, $T\rightarrow S$ is (flat and) convex by corollary~\ref{cor:SmoothOverPruferIsConvexe},
so $\mathbb{G}_{m,S}^{lft}(T)\rightarrow\mathbb{G}_{m,\eta}(T_{\eta})$
is an isomorphism. This says that any morphism of $\eta$-schemes
$a:G_{\eta}\rightarrow\mathbb{G}_{m,\eta}$ extends uniquely to a
morphism of $S$-schemes $\alpha:G\rightarrow\mathbb{G}_{m,S}^{lft}$.
If $a$ is a character, i.e.~compatible with the group structures,
then so is $\alpha$. For instance to check that $\alpha(xy)=\alpha(x)\alpha(y)$,
we may consider the corresponding scheme of coincidence $Z$ in $G\times_{S}G$;
since $\mathbb{G}_{m,S}^{lft}$ is separated over $S$ by proposition~\ref{prop:SeparatedCase},
$Z$ is closed in $G\times_{S}G$; since $a=\alpha_{\eta}$ is a character,
$Z_{\eta}=(G\times_{S}G)_{\eta}$; since $G\times_{S}G$ is flat over
$S$, $Z=G\times_{S}G$, which proves our claim. Now observe that
since $G\rightarrow S$ is finitely presented, its fibers have finite
groups of connected components; on the other hand since $S$ is normal,
the groups of connected components of the fibers of $\mathbb{G}_{m,S}^{lft}\rightarrow S$
have no torsion; thus $\alpha$ factors through the neutral component
$\mathbb{G}_{m,S}$ of $\mathbb{G}_{m,S}^{lft}$. So any character
$a:G_{\eta}\rightarrow\mathbb{G}_{m,\eta}$ indeed extends uniquely
to a character $\alpha:G\rightarrow\mathbb{G}_{m,S}$.
\end{proof}
\bibliographystyle{plain}
\bibliography{MyBib}

\end{document}